\documentclass[a4paper]{amsart}

\newcommand{\cB}{\mathcal B}

\newcommand{\N}{\mathbb{N}}

\newcommand{\Z}{\mathbb{Z}}

\newcommand{\set}[1]{\left\{ #1 \right\}}

\newtheorem{thm}{Theorem}[section]
\newtheorem{thm*}{Theorem}
\newtheorem{theorem}{Theorem}[section]
\newtheorem{theorem*}{Theorem}

\newtheorem{lem}[thm]{Lemma}

\newtheorem{proposition}[thm]{Proposition}

\theoremstyle{definition}

\newtheorem{defn*}{Definition}
\newtheorem{definition}{Definition}[section]
\newtheorem{definition-st}{Definition}

\newtheorem{example}{Example}[section]
\newtheorem{example-st}{Example}

\newtheorem{rem-st}{Remark}
\newtheorem{remark}{Remark}[section]
\newtheorem{remark-st}{Remark}

\numberwithin{equation}{section}
\numberwithin{figure}{section}

\newtheorem{cor*}{Corollary}
\newtheorem{prop*}{Proposition}

\newtheorem{obser-st}{Observation}

\newtheorem{question}[thm]{Question}

\theoremstyle{definition}

\newtheorem{exmp-st}{Example}

\newtheorem{exmps-st}{Examples}

\theoremstyle{remark}

\newtheorem{rems-st}{Remarks} 

\newtheorem{ack-st}{Acknowledgment}

\newcommand{\Words}{\cB}

\begin{document}
\title{On continuing codes}
\author{Jisang Yoo}
\address{Ajou University, Suwon, South Korea}
\email{(replace X with my last name) jisang.X.ac+rs12@gmail.com}

\begin{abstract}
We investigate what happens when we try to work with continuing block codes (i.e. left or right continuing factor maps) between shift spaces that may not be shifts of finite type. For example, we demonstrate that continuing block codes on strictly sofic shifts do not behave as well as those on shifts of finite type; a continuing block code on a sofic shift need not have a uniformly bounded retract, unlike one on a shift of finite type. A right eresolving code on a sofic shift can display any behavior arbitrary block codes can have. We also show that a right continuing factor of a shift of finite type is always a shift of finite type.
\end{abstract}


\maketitle
\tableofcontents

\section{Introduction}

There has been many studies on sliding block codes between shift spaces. They form bases for further studies on topics such as hidden Markov chains, cellular automata, and relative thermodynamic formalism. Sliding block codes on shifts of finite type can be divided into two categories: finite-to-one codes and infinite-to-one codes. The simplest class in the first category is the class of right closing codes. Many things are known about right closing codes between shift spaces \cite{BMT-ResolvingDimensionGroup,LM}.

Right continuing codes are a natural generalization of right resolving codes and right closing codes. Right continuing codes are infinite-to-one analog of right closing codes.

Right continuing codes are less well understood than right closing codes are, but some nice properties of such codes are known. For example, continuing codes between SFTs work nice with Markov measures;
if $\phi$ is a continuing code from an SFT $X$ onto another SFT $Y$, each Markov measure on $Y$ lifts to uncountably many Markov measures on $X$ (see \cite{BT-ITO}); 
if $\phi$ is a fiber-mixing code from an SFT $X$ onto another SFT $Y$, i.e., given $x, \bar{x} \in X$ with $\phi (x) = \phi (\bar{x}) = y \in Y$, there is $z \in \phi^{-1} (y)$ that is left asymptotic to $x$ and right asymptotic to $\bar{x}$, then $\phi$ projects each Markov measure on $X$ onto a Gibbs measure on $Y$ with a H\"older continuous potential (see \cite{CU-ProjectionGibbsian, yoogibbs}). (Fiber-mixing condition implies that $\phi$ is left continuing and right continuing)
It is also known that one can always construct a right continuing code from a given irreducible sofic shift onto a given irreducible SFT of lower entropy (see \cite{J-ExistenceBiContinuing}).

In this paper, we investigate how right continuing codes behave with respect to shift spaces that are not SFTs. We show that right continuing codes defined on general shift spaces don't necessarily have a retract. But when it does have a retract, we show that it can be assumed to be zero via recoding.
We show that there is no right continuing codes from an SFT onto a non-SFT.
We also learn that the property of right eresolving codes on non-SFTs are not strong enough to tame block codes in any meaningful way, in particular, the property of  right eresolving doesn't even imply the property of right continuing.

We hope that this work will be helpful in understanding right continuing codes and infinite-to-one codes on general shift spaces.

\section{Preliminary}

An \emph{SFT} or a shift of finite type is a shift space determined by a finite set of forbidden words. An N-step SFT is an SFT determined by a set of forbidden words of length $N+1$. It is known that any SFT is topologically conjugate to a 1-step SFT.  A \emph{non-SFT} is a shift space that is not an SFT. The shift map will be denoted as $\sigma$.

A \emph{(sliding) block code} is a homomorphism between two topological dynamical systems that are shift spaces. A block code $\phi$ is a 1-block code if $(\phi(x))_0$  is determined by $x_0$. It is known that any block code can be assumed to be a 1-block code without loss of generality via recoding, i.e. any block code is topologically conjugate to a 1-block code. It is also known that any block code on a SFT can be recoded to be a 1-block code on a 1-step SFT. A block code from $X$ \emph{onto} $Y$ is called a \emph{factor map} between $X$ and $Y$, and $Y$ is called a \emph{factor} of $X$. Factors of SFTs are called \emph{sofic shifts}.

\begin{definition}
  The language $\Words(X)$ of a shift space $X$ is the collection of words appearing in points of $X$. The set $\Words_n(X)$ is defined to be the collection of words of length $n$ appearing in points of $X$. In other words, $\Words_n(X) = \{x_{[0,n-1]} : x \in X\}$.
\end{definition}

\begin{definition}
  Two points $x,x'$ in a shift space are \emph{left asymptotic} (resp. \emph{right asymptotic}) to each other if there is $n\in\Z$ with $x_{(-\infty,n]} = x'_{(-\infty,n]}$ (resp. $ x_{[n,\infty)} = x'_{[n,\infty)} $), or equivalently if $d(\sigma^{-k}x, \sigma^{-k}x') \to 0$ (resp. $d(\sigma^kx, \sigma^kx') \to 0$) as $k \to \infty$.
\end{definition}

\begin{definition}
  A 1-block code $\phi$ from a shift space $X$ onto another shift space $Y$ is \emph{right eresolving} if for $a_0 \in \Words_1(X)$ and $b_0b_1 \in \Words_2(Y)$ with $\phi(a_0) = b_0$, there is at least one $a_1$ with $a_0a_1\in\Words_2(X)$ and $\phi(a_1)=b_1$.
\end{definition}

\begin{definition}
  A block code $\phi$ from a shift space $X$ onto another shift space $Y$ is \emph{right continuing} or \emph{u-eresolving} if for each $x \in X$ and $y\in Y$ such that $\phi(x)$ is left-asymptotic to $y$, there is at least one $\bar x \in X$ such that $\bar x$ is left asymptotic to $x$ and $\phi(\bar x)=y$.
\end{definition}

The terms \emph{left eresolving, left continuing} and \emph{s-eresolving} are defined analogously.

\begin{definition}
  A  block code from a shift space onto another shift space is \emph{continuing} if it is left continuing or right continuing.
\end{definition}

\begin{definition}
  A right continuing code $\phi$ from a shift space $X$  onto another shift space $Y$ is said to have a \emph{retract} $n \in \N$ if for all $x\in X$ and $y \in Y$
  with $(\phi x)_i = y_i$ for all $i \le 0$, there is $\bar x \in X$
  with $\phi\bar x = y$ and $\bar x_i = x_i$ for all $i \le -n$.
\end{definition}

We mention a result \cite{BT-ITO} on a relationship between right eresolving codes and right continuing codes on SFTs. (The term \emph{right resolving} in the paper \cite{BT-ITO} corresponds to the term \emph{right eresolving} in this paper)

\begin{proposition}\label{prop:right-continuing}
  For a factor map $\phi$ between SFTs, the following three conditions are equivalent. 
  \begin{enumerate}
  \item $\phi$ is topologically equivalent to a right eresolving 1-block factor map between 1-step SFTs.
  \item $\phi$ is right continuing.
  \item $\phi$ is right continuing and has a retract.
  \end{enumerate}
\end{proposition}

After we prove that the right continuing factor of an SFT is always an SFT, the requirement of the codomain being an SFT from the above proposition can be safely dropped.

\section{A right continuing code without a retract}

This example and the proof are also included in \cite{J-ExistenceBiContinuing} by personal communication.\footnote{That paper needed the result before this paper's draft were written.}

\begin{example}
  Let $X$ be a shift space with the alphabet $\{1,\bar{1},2,3\}$
  defined by forbidden blocks $\{\bar{1}2^n 3 : n \ge 0 \}$ and let
  $Y$ be the full shift $\{1,2,3\}^\mathbb{Z}$. Then $X$ is an irreducible sofic shift. Let $\phi : X \to Y$
  be the 1-block code defined by $\phi(1) = \phi(\bar{1}) = 1,\,
  \phi(2)=2, \, \phi(3)=3$. Then the map $\phi$ is a right-continuing factor map without a retract.
\end{example}
  \begin{proof}
    The map $\phi$ is clearly onto $Y$.

    Next, we need to check that it is right continuing. Suppose $x\in X$ and $y\in Y$ with
    $\phi(x_i)=y_i$ for all $i\le 0$.  If $x^{(1)} :=
    x_{(-\infty,-1]}.x_0y_{[1,\infty)}$ is in $X$, then we are done
    because $x^{(1)}$ is left-asymptotic to $x$ and $\phi(x^{(1)}) =
    y$.  Otherwise, the word $\bar{1}2^n 3$ for some $n \ge 0$ occurs
    exactly once in $x^{(1)}$.  Let $x^{(2)}$ be the point obtained
    from $x^{(1)}$ by replacing all occurrences of $\bar{1}2^n 3$ with $12^n 3$.  Then the
    point $x^{(2)}$ is in $X$, is left-asymptotic to $x$ and satisfies
    $\phi(x^{(2)}) = y$. This completes the proof of the right continuing property.

    Finally, we check that the map $\phi$ has no retract. This follows from considering, for each $n\ge0$, the points $x =
    \bar{1}^\infty 2^n . 2 2^\infty \in X$ and $y = 1^\infty 2^n . 2
    3^\infty \in Y$.
  \end{proof}

\section{Right eresolving codes on non-SFTs}

Consider an arbitrary 1-block code $\phi$ from a shift space $X$ onto another shift space $Y$. One can transform $\phi: X \to Y$ into a right eresolving code $\sqrt \phi : \sqrt X \to \sqrt Y$ without losing too much of the original code's properties.

First, we choose a symbol $a$  disjoint from the alphabets of $X$ and $Y$. Then define shift spaces $\sqrt X = \set{\eta(x), \sigma(\eta(x)) : x \in X } $ and  $\sqrt Y = \set{\eta(y), \sigma(\eta(y)) : y \in Y } $ where $\eta$ is defined by $ \eta(x)_{2i} = x_i$ and $\eta(x)_{2i+1} = a $ for all $i \in \Z$. 
Then define a 1-block code $\sqrt \phi$ from $\sqrt X$ onto $\sqrt Y$ by 

\[ \sqrt\phi(c) = 
\begin{cases}
  a & \text{if $c=a$,} \\
  \phi(a) & \text{if $c\neq a$}
\end{cases}
\]

for $c \in \Words_1(\sqrt X)$.

The map $\sqrt \phi$ is easily seen to be right eresolving and shares many properties with the original $\phi$. For example, if $\phi$ is not right continuing, then neither is $\sqrt \phi$. Thus right eresolving codes are not necessarily right continuing. This is in contrast with Proposition~\ref{prop:right-continuing} in the case of SFTs.

\section{A right continuing factor of an SFT is an SFT}

A right eresolving factor of 1-step SFT is easily seen to be a 1-step SFT. Since a right continuing codes between SFTs can be assumed to be right eresolving by recoding (see Proposition~\ref{prop:right-continuing}), one might think that this would easily imply that the right continuing factor of an SFT is an SFT. But the proof of Proposition~\ref{prop:right-continuing} uses the assumption of $Y$ being an SFT, which in turns is what we need to prove. Nevertheless, we manage to prove the following result.

\begin{thm}\label{thm:rc-sft-sft}
 If $X$ is an SFT and $\phi$ a right continuing code onto a shift space $Y$. Then $Y$ is an SFT.
\end{thm}

Throughout this section, we assume that $\phi$ is a right continuing 1-block code from a 1-step SFT $X$ onto a shift space $Y$ unless stated otherwise. We cannot assume without loss of generality that $\phi$ is eresolving to make things easier at this point. We will not assume the retract to be 0 because we want to see how the memory of $Y$ grows depending on the size of the retract. We prove the theorem through a series of lemmas.

\begin{lem}\label{lem:i}
  There is $R \in \mathbb{N}$ such that for all $x\in X$ and $y \in Y$
  with $\phi(x_i)=y_i,\, i\le0$, there exist $x'_{[-R,1]} \in \mathcal
  B (X)$ with $x'_{-R} = x_{-R}$ and $\phi(x'_i) = y_i,\, i\in
  [-R,1]$.
\end{lem}
\begin{proof}
  Suppose it is false. Then for all $n \in \mathbb N$, there exist $x^{(n)}
  \in X$ and $y^{(n)} \in Y$ such that $\phi(x^{(n)}_i) = y^{(n)}_i,\,
  i\le0$ and that there does not exist a word $x'_{[-n,1]} \in
  \mathcal B (X)$ satisfying 

  \[(*) \quad x^{(n)}_{-n} = x'_{-n},\quad
  \phi(x'_{[-n,1]}) = y^{(n)}_{[-n,1]}.\]

  By compactness argument, there is a limit point $(x,y) \in X\times Y$  of the sequence of pairs
  $(x^{(n)}, y^{(n)})$. We have $\phi(x_i) = y_i,\, i\le0$ and since
  $\phi$ is right continuing, there are $\bar n \in \mathbb N$ and
  $\bar x \in X$ such that $\phi(\bar x) = y$ and $\bar x_i = x_i,\,
  i\le -\bar n$. We can choose $n\in \mathbb N$ such that $n > \bar n$
  and $x_i = x^{(n)}_i,\, y_i = y^{(n)}_i,\, i\in[-\bar n, 1]$.  But
  the word $x'_{[-n,1]} := x^{(n)}_{[-n,-\bar n -1]}\bar x_{[-\bar
    n,1]}$, which is in $\Words(X)$ because $X$ is a 1-step SFT, satisfies (*) and that contradicts our initial assumption.
\end{proof}

Since $X$ is assumed to be a 1-step SFT, it is easily shown by repeated application of the previous lemma that the number $R$ is a retract of $\phi$.

\begin{lem}
  There is $K \in \mathbb N$ such that for all $\pi \in \mathcal
  B_K(X)$ and $u \in \mathcal B_{K+1}(Y)$ with $\phi(\pi)$ being a prefix
  of $u$, there exist $\pi' \in \mathcal B_{K+1}(X)$ with
  $\pi_1 = \pi'_1$ and $\phi(\pi') = u$.
\end{lem}
\begin{proof}
  Let $d := (|\mathcal B_1(X)|)^2 +1$ and $ K := d+R+1$ and consider two words
  $x_{[-R-d,0]} \in \mathcal B_K(X)$ and $ y_{[-R-d,1]} \in \mathcal
  B_{K+1}(Y)$ with $\phi(x_{[-R-d,0]}) = y_{[-R-d,0]}$.

  We can choose a $\phi$-preimage $\bar x_{[-R-d,1]} \in \mathcal B(X)$
  of $y_{[-R-d,1]}$. By the pigeon hole principle, there exist $-R-d\le I
  < J \le -R-1$ such that $x_I = x_J$ and $ \bar x_I = \bar x_J$. Build
  a sequence $x_1x_2x_3\cdots$ that extends $x_{[-R-d,0]}$ and $\bar
  x_2 \bar x_3\cdots$ that extends $\bar x_{[-R-d,1]}$.  Now
  $x_{[-R-d,\infty)}$ and $ \bar x_{[-R-d,\infty)}$ are right-infinite
  sequences allowed in $X$.  The point $\hat x :=
  (x_{[I,J-1]})^\infty x_{[J,-1]}.x_{[0,\infty)}$ is in $X$. The
  point $(\bar x_{[I,J-1]})^\infty\bar x_{[J,-1]}.\bar x_{[0,\infty)}$
  is also in $X$ and its image $\hat y :=
  (y_{[I,J-1]})^\infty y_{[J,-1]}.y_0y_1\phi(\bar x_2)\phi(\bar
  x_3)\cdots$ is in $Y$ and we have $\phi(\hat x_i) = \hat y_i,\,
  i\le0$.  By applying the previous lemma to $\hat x$ and $\hat y$, we can see that
  there exist $x'_{[-R,1]} \in \mathcal B(X)$ satisfying $x'_{-R} =
  \hat x_{-R} = x_{-R}$ and $\phi(x'_{[-R,1]}) = \hat y_{[-R,1]} =
  y_{[-R,1]}$.

  Define $x'_{[-R-d,-R-1]} := x_{[-R-d,-R-1]}$, then we have
  $x'_{[-R-d,1]} \in \mathcal B(X)$ and $\phi(x'_{[-R-d,1]}) =
  y_{[-R-d,1]}$.
\end{proof}
\begin{remark}
  In order to prove above lemma, one might be temped to just set $ K := R $ and start by extending the word $\pi$ to a point $x \in X$ such that $ x_{[-R,0]} = \pi $  and then somehow extend $y$ to apply Lemma~\ref{lem:i}. But choosing an appropriate $y \in Y$ requires showing that $ x_{[-\infty, -R-1]} u \in \Words(Y) $, which you cannot show at this point because $ Y $ is not yet known to be an SFT, let alone an $R$-step SFT.
\end{remark}

\begin{lem}
  For words $u,v \in \Words(Y)$ and a symbol $a \in \Words_1(Y)$ such that $uv \in \mathcal B(Y),\, va \in \mathcal
  B(Y)$ and $ |v| = K$, we have $uva \in \mathcal B(Y)$
\end{lem}
\begin{proof}
  There exist $\pi,\tau \in \Words(X)$ such that $\pi\tau\in\mathcal B(X),\,
  \phi(\pi\tau) = uv$ and $ |\tau| = K$.  Applying the previous lemma to $\tau$ and
  $va$, we see that there is $\tau' \in \mathcal B(X)$ with $\tau_1 =
  \tau'_1$ and $\phi(\tau') = va$. Clearly $\pi\tau' \in \mathcal
  B(X)$ and $\phi(\pi\tau') = uva$.
\end{proof}

\begin{proposition}
  $Y$ is a $K$-step SFT.
\end{proposition}
\begin{proof}
  It follows easily from the previous lemma.
\end{proof}

The proof of Theorem~\ref{thm:rc-sft-sft} is now complete. 

\begin{remark}
  Since $Y$ is now shown to be an SFT, Proposition~\ref{prop:right-continuing} applies and therefore $\phi$ is topologically equivalent to a right-eresolving code from a 1-step SFT onto another 1-step SFT.
\end{remark}


\begin{remark}
  This result implies  in particular that there is no right continuing
  code from an SFT onto a strictly sofic shift. This contrasts with
  the fact that every sofic shift is a right closing factor of an SFT.
\end{remark}

\begin{question}
  We showed that $Y$ is a $K$-step SFT where $K = R+2+(|\Words_1(X)|)^2$. Is there a better bound on $K$?
\end{question}

\section{Decreasing the retract of a right continuing map}

\begin{theorem}
  Given $\phi$ a right continuing code with a retract from a shift space $X$ onto a shift space $Y$, there exist a topological conjugacy $\psi$ from $X$ to a shift space $\bar X$ such that $\phi \circ \psi^{-1} : \bar X \to Y$ is a right continuing 1-block code with retract 0.
\end{theorem}
\begin{proof}
  We may assume that $\phi$ is a 1-block code and its retract is $R$. 

  Define a block code $\psi$ on $X$ by
  \[ (\psi x)_i = (x_{i-R}, (\phi x)_{[i-R,i]}) \] for all $x \in X$ and $i \in \Z$. This code is injective, therefore a conjugacy onto its image $\bar X$.
  
  Let $\bar{\phi} := \phi \circ \psi^{-1}$. This is a 1-block code because for $\bar{x} \in \bar{X}$ and $x := \psi^{-1}(\bar{x})$ we have
  \begin{align*}
    (\bar\phi \bar x)_0 &= (\phi x)_0 \\
    &= F((\psi x)_0) \\
    &= F(\bar x_0)
  \end{align*}
  where $F$ is a function that maps $(a, b_0 b_1 \ldots b_R)$ to $b_R$.
  
  It remains to show that $\bar\phi$ is right continuing with retract 0.
  
  Suppose that we are given $\bar x \in \bar X$, $y \in Y$ with $(\bar\phi\bar x)_{[-\infty,0]} = y_{[-\infty,0]} $.
  Then the point $x := \psi^{-1} \bar x$ satisfies $(\phi x)_{[-\infty,0]} = y_{[-\infty,0]}$.
  But since $\phi$ has retract $R$, there is $z \in X$ satisfying $x_{[-\infty,-R]} = z_{[-\infty,-R]}$ and $\phi z = y$.
  The point $\bar z := \psi z$ satisfies the desired properties: for $i \le 0$
  \begin{align*}
    \bar z_i &= (z_{i-R}, (\phi z)_{[i-R,i ]}) \\
    &= (z_{i-R}, y_{[i-R,i]}) \\
    &= (x_{i-R}, (\phi x)_{[i-R,i]}) \\
    &= \bar x_i
  \end{align*}
  and $\bar \phi \bar z = y$.
\end{proof}


For the next theorem, we need some definitions.

\begin{definition}
  A  block code from a shift space onto another shift space is \emph{bi-continuing} if it is left continuing \emph{and} right continuing.
\end{definition}

\begin{definition}
  A bi-continuing code $\phi$ from a shift space $X$  onto another shift space $Y$ is said to have \emph{retracts} $n_1, n_2 \in \N$ if $n_1$ is a retract of $\phi$ as a right continuing code and if $n_2$ is a retract of $\phi$ as a left continuing code.
\end{definition}

\begin{theorem} 
  Given a bi-continuing factor map $ \phi $  with retracts from a shift space $X$ to a shift space $Y$,
  there are conjugacies $ \psi : X \to \bar X $ and $ \theta : Y \to \bar Y $
  such that $ \theta \circ \phi \circ \psi^{-1} : \bar X \to \bar Y $ is a bi-continuing 1-block factor map with retracts 0, 0.
\end{theorem}
\begin{proof}
  We may assume $\phi$ is a 1-block code with retracts $R, R$.
  Let $\bar X$ and $\bar Y$ be images of block codes $\psi$ and $\theta$ defined by
  \begin{align*}
    (\psi x)_i &= (x_i, (\phi x)_{[i-R,i+R]}) \\
    (\theta y)_i &= y_{[i-R,i+R]}
  \end{align*}
  for $i\in \Z$, $x \in X$ and $y \in Y$.
  They are easily checked to be injective, hence they are conjugacies.
  
  Let $ \bar\phi := \theta \circ \phi \circ \psi^{-1} : \bar X \to \bar Y $, then it is a 1-block code since for all $ \bar x \in \bar X $ and $ x := \psi^{-1}\bar x $ we have
  \begin{align*}
    (\bar\phi\bar x)_0 &= (\theta \phi x)_0 \\
    &= (\phi x)_{[i-R,i+R]} \\
    &= F((\psi x)_0) \\
    &= F(\bar x_0)
  \end{align*}
  where the function $F$ maps $ (a,b) $ to $b$.

  To show that $\bar \phi$ is right continuing with retract 0, suppose we are given $ \bar x \in \bar X $ and $ \bar y \in \bar Y $ with $ (\bar \phi \bar x)_{[-\infty,0]} = \bar y_{[-\infty,0]} $ and let $ x = \psi^{-1} \bar x $ and $ y = \theta^{-1} \bar y $. Then we have $ y_{[-\infty,R]} = (\phi x)_{[-\infty,R]} $ and since $ \phi $ is right continuing with retract $R$, there is $ z \in X $ such that $ x_{[-\infty,0]} = z_{[-\infty,0]} $  and $ \phi z = y $. Let $ \bar z = \psi z $, then we have for $i \le 0$, 
  \begin{align*}
    \bar x_i &= (x_i, (\phi x)_{[i-R,i+R]}) \\
    &= (x_i, y_{[i-R,i+R]}) \\
    &= (z_i, (\phi z)_{[i-R,i+R]}) \\
    &= \bar z_i
  \end{align*}
  and that $ \bar \phi \bar z = \bar y $ 

  Similarly, $\bar \phi$ is also left continuing with retract 0.
\end{proof}

\bibliographystyle{amsplain}
\providecommand{\bysame}{\leavevmode\hbox to3em{\hrulefill}\thinspace}
\providecommand{\MR}{\relax\ifhmode\unskip\space\fi MR }
\providecommand{\MRhref}[2]{%
  \href{http://www.ams.org/mathscinet-getitem?mr=#1}{#2}
}
\providecommand{\href}[2]{#2}

\end{document}